\documentclass{IEEEtran}

\usepackage{amsmath}
\usepackage{amssymb}
\usepackage{cite}
\usepackage{graphicx}
\usepackage{verbatim}
\usepackage{bbm}
\usepackage{subfigure}
\usepackage{url}

\newcommand{\mb}[1]{{\bf #1}}

\renewcommand{\a}{\mb{a}}
\renewcommand{\b}{\mb{b}}

\renewcommand{\r}{\mb{r}}

\newcommand{\x}{\mb{x}}
\newcommand{\w}{\mb{w}}
\newcommand{\y}{\mb{y}}
\newcommand{\z}{\mb{z}}

\newcommand{\A}{\mb{A}}

\renewcommand{\H}{\mb{H}}
\newcommand{\I}{\mb{I}}

\newcommand{\M}{\mb{M}}

\newcommand{\zero}{\mb{0}}
\newcommand{\one}{\mathbbm{1}}

\newcommand{\beq}{\begin{equation}}
\newcommand{\eeq}{\end{equation}}

\newcommand{\hx}{{\widehat{\x}}}

\newcommand{\xo}{{\hat{\x}_{\mathrm{or}}}}
\newcommand{\xoL}{{\hat{\x}_{\mathrm{or},\Lambda}}}
\newcommand{\xzL}{{\x_{0,\Lambda}}}
\newcommand{\xbp}{{\hat{\x}_\mathrm{BP}}}
\newcommand{\xds}{{\hat{\x}_\mathrm{DS}}}

\newcommand{\xth}{{\hx_{\mathrm{th}}}}
\newcommand{\xomp}{{\hx_{\mathrm{OMP}}}}
\newcommand{\xmin}{{|x_{\min}|}}
\newcommand{\xmax}{{|x_{\max}|}}
\newcommand{\xzLz}{{\x_{0,\Lz}}}

\newcommand{\AL}{\A_\Lambda}
\newcommand{\ALz}{\A_\Lz}

\newcommand{\Lz}{{\Lambda_0}}
\newcommand{\st}{\text{s.t. }}
\newcommand{\Fchi}{{F_{\chi^2_s}}}

\newcommand{\eps}{\varepsilon}

\newcommand{\RR}{{\mathbb R}}

\newcommand{\MSE}{{\mathrm{MSE}}}
\newcommand{\CRB}{{\mathrm{CRB}}}

\newcommand{\E}[1]{E \! \left\{ #1 \right\}}
\renewcommand{\Pr}[1]{\mathrm{Pr} \kern -1pt \left\{ #1 \right\}}

\newcommand{\pinv}{\dagger}

\newtheorem{theorem}{Theorem}
\newtheorem{lemma}{Lemma}

\newtheorem{corollary}{Corollary}

\DeclareMathOperator{\Tr}{Tr}

\DeclareMathOperator{\spn}{span}

\DeclareMathOperator{\supp}{supp}

\newcommand{\Qfunc}[1]{Q \! \left( #1 \right)}

\begin{document}

\title{Coherence-Based Performance~Guarantees for Estimating a~Sparse~Vector Under Random~Noise}
\author{Zvika~Ben-Haim, Yonina~C.~Eldar, and Michael~Elad\thanks{Z.~Ben-Haim and Y.~C.~Eldar are with the Department of Electrical Engineering, Technion---Israel Institute of Technology, Haifa 32000, Israel (e-mail: \{zvikabh@tx, yonina@ee\}.technion.ac.il). M.~Elad is with the Department of Computer Science, Technion---Israel Institute of Technology, Haifa 32000, Israel (e-mail: elad@cs.technion.ac.il). Contact information for Z.~Ben-Haim: phone +972-4-8294700, fax +972-4-8295757. This work was supported in part by the Israel Science Foundation under Grants 1081/07 and 599/08, by the European Commission in the framework of the FP7 Network of Excellence in Wireless COMmunications NEWCOM++ (contract no. 216715), and by the Goldstein UAV and Satellite Center.}}
\maketitle

\begin{abstract}
We consider the problem of estimating a deterministic sparse vector $\x_0$ from underdetermined measurements $\A\x_0 + \w$, where $\w$ represents white Gaussian noise and $\A$ is a given deterministic dictionary. We analyze the performance of three sparse estimation algorithms: basis pursuit denoising (BPDN), orthogonal matching pursuit (OMP), and thresholding. These algorithms are shown to achieve near-oracle performance with high probability, assuming that $\x_0$ is sufficiently sparse. Our results are non-asymptotic and are based only on the coherence of $\A$, so that they are applicable to arbitrary dictionaries. Differences in the precise conditions required for the performance guarantees of each algorithm are manifested in the observed performance at high and low signal-to-noise ratios. This provides insight on the advantages and drawbacks of $\ell_1$ relaxation techniques such as BPDN as opposed to greedy approaches such as OMP and thresholding.
\end{abstract}

{\it EDICS Topics:} SSP-PARE, SSP-PERF.


\section{Introduction}

Estimation problems with sparsity constraints have attracted considerable attention in recent years because of their potential use in numerous signal processing applications, such as denoising, compression and sampling. In a typical setup, an unknown deterministic parameter $\x_0 \in \RR^m$ is to be estimated from measurements $\b = \A\x_0 + \w$, where $\A \in \RR^{n \times m}$ is a deterministic matrix and $\w$ is a noise vector. Typically, the dictionary $\A$ consists of more columns than rows (i.e., $m>n$), so that without further assumptions, $\x_0$ is unidentifiable from $\b$. The impass\'e is resolved by assuming that the parameter vector is sparse, i.e., that most elements of $\x_0$ are zero. Under the assumption of sparsity, several estimation approaches can be used. These include greedy algorithms, such as thresholding and orthogonal matching pursuit (OMP) \cite{pati93}, and $\ell_1$ relaxation methods, such as the Dantzig selector \cite{candes07} and basis pursuit denoising (BPDN) \cite{chen98} (also known as the Lasso). A comparative analysis of these techniques is crucial for determining the appropriate strategy in a given situation.

There are two standard approaches to modeling the noise $\w$ in the sparse estimation problem. The first is to assume that $\w$ is deterministic and bounded \cite{donoho06, candes06, fuchs05}. This leads to a worst-case analysis in which an estimator must perform adequately even when the noise maximally damages the measurements. The noise in this case is thus called adversarial. By contrast, if one assumes that the noise is random, then the analysis aims to describe estimator behavior for typical noise values \cite{tropp06, candes07, bickel08}. The random noise scenario is the main focus of this paper. As one might expect, stronger performance guarantees can be obtained in this setting.

It is common to judge the quality of an estimator by comparing its mean-squared error (MSE) with the Cram\'er--Rao bound (CRB) \cite{kay93}. In the case of sparse estimation under Gaussian noise, it has recently been shown that the unbiased CRB is identical (for almost all values of $\x_0$) to the MSE of the ``oracle'' estimator, which knows the locations of the nonzero elements of $\x_0$ \cite{ben-haim09b}. Thus, a gold standard for estimator performance is the MSE of the oracle. Indeed, it can be shown that $\ell_1$ relaxation algorithms come close to the oracle when the noise is Gaussian. Results of this type are sometimes referred to as ``oracle inequalities.'' Specifically, Cand\`es and Tao \cite{candes07} have shown that, with high probability, the $\ell_2$ distance between $\x_0$ and the Dantzig estimate is within a constant times $\log m$ of the performance of the oracle. Recently, Bickel et al.\ \cite{bickel08} have demonstrated that the performance of BPDN is similarly bounded, with high probability, by $C \log m$ times the oracle performance, for a constant $C$. However, the constant involved in this analysis is considerably larger than that of the Dantzig selector. Interestingly, it turns out that the $\log m$ gap between the oracle and practical estimators is an unavoidable consequence of the fact that the nonzero locations in $\x_0$ are unknown \cite{candes06b}.

The contributions \cite{candes07, bickel08} state their results using the restricted isometry constants (RICs). These measures of the dictionary quality can be efficiently approximated in specific cases, e.g., when the dictionary is selected randomly from an appropriate ensemble. However, in general it is NP-hard to evaluate the RICs for a given matrix $\A$, and they must then be bounded by efficiently computable properties of $\A$, such as the mutual coherence \cite{donoho01}. In this respect, coherence-based results are appealing since they can be used with arbitrary dictionaries \cite{CandesPlan09, cai09}.

In this paper, we seek performance guarantees for sparse estimators based directly on the mutual coherence of the matrix $\A$ \cite{ben-haim10b}. While such results are suboptimal when the RICs of $\A$ are known, the proposed approach yields tighter bounds than those obtained by applying coherence bounds to RIC-based results. Specifically, we demonstrate that BPDN, OMP and thresholding all achieve performance within a constant times $\log m$ of the oracle estimator, under suitable conditions. In the case of BPDN, our result provides a tighter guarantee than the coherence-based implications of the work of Bickel et al.\ \cite{bickel08}. To the best of our knowledge, there are no prior performance guarantees for greedy approaches such as OMP and thresholding when the noise is random.

It is important to distinguish the present work from Bayesian performance analysis, as practiced in \cite{schnass07, tropp08, CandesPlan09, eldar09c}, where on top of the assumption of stochastic noise, a probabilistic model for $\x_0$ is also used. Our results hold for any specific value of $\x_0$ (satisfying appropriate conditions), rather than providing results on average over realizations of $\x_0$; this necessarily leads to weaker guarantees. It also bears repeating that our results apply to a fixed, finite-sized matrix $\A$; this distinguishes our work from asymptotic performance guarantees for large $m$ and $n$, such as \cite{wainwright06}.

The rest of this paper is organized as follows. We begin in Section~\ref{se:pre} by comparing dictionary quality measures and reviewing standard estimation techniques. In Section~\ref{se:adversarial}, we analyze the limitations of estimator performance under adversarial noise. This motivates the introduction of random noise, for which substantially better guarantees are obtained in Section~\ref{se:random}. Finally, the validity of these results is examined by simulation in practical estimation scenarios in Section~\ref{se:numer}.

The following notation is used throughout the paper. Vectors and matrices are denoted, respectively, by boldface lowercase and boldface uppercase letters. The set of indices of the nonzero entries of a vector $\x$ is called the support of $\x$ and denoted $\supp(\x)$. Given an index set $\Lambda$ and a matrix $\A$, the notation $\AL$ refers to the submatrix formed from the columns of $\A$ indexed by $\Lambda$. The $\ell_p$ norm of a vector $\x$, for $1 \le p \le \infty$, is denoted $\|\x\|_p$, while $\|\x\|_0$ denotes the number of nonzero elements in $\x$.

\section{Preliminaries}
\label{se:pre}

\subsection{Characterizing the Dictionary}
\label{ss:dict}

Let $\x_0 \in \RR^m$ be an unknown deterministic vector, and denote its support set by $\Lz = \supp(\x_0)$. Let $s = \|\x_0\|_0$ be the number of nonzero entries in $\x_0$. In our setting, it is typically assumed that $s$ is much smaller than $m$, i.e., that most elements in $\x_0$ are zero. Suppose we obtain noisy measurements
\beq \label{eq:b=Ax+w}
\b = \A\x_0 + \w
\eeq
where $\A \in \RR^{n \times m}$ is a known overcomplete dictionary ($m>n$). We refer to the columns $\a_i$ of $\A$ as the \emph{atoms} of the dictionary, and assume throughout our work that the atoms are normalized, $\|\a_i\|_2 = 1$. We will consider primarily the situation in which the noise $\w$ is random, though for comparison we will also examine the case of a bounded deterministic noise vector; a precise definition of $\w$ is deferred to subsequent sections.

For $\x_0$ to be identifiable, one must guarantee that different values of $\x_0$ produce significantly different values of $\b$. One way to ensure this is to examine all possible \emph{subdictionaries}, or  $s$-element sets of atoms, and verify that the subspaces spanned by these subdictionaries differ substantially from one another.

More specifically, several methods have been proposed to formalize the notion of the suitability of a dictionary for sparse estimation. These include the mutual coherence \cite{donoho01}, the cumulative coherence \cite{tropp06}, the exact recovery coefficient (ERC) \cite{tropp06}, the spark \cite{donoho06}, and the RICs \cite{candes06, candes07}. Except for the mutual coherence and cumulative coherence, none of these measures can be efficiently calculated for an arbitrary given dictionary $\A$. Since the values of the cumulative and mutual coherence are quite close, our focus in this paper will be on the mutual coherence $\mu = \mu(\A)$, which is defined as
\beq \label{eq:def mu}
\mu \triangleq \max_{i \ne j} \left| \a_i^T \a_j \right|.
\eeq

While the mutual coherence can be efficiently calculated directly from \eqref{eq:def mu}, it is not immediately clear in what way $\mu$ is related to the requirement that subdictionaries must span different subspaces. Indeed, $\mu$ ensures a lack of correlation between single atoms, while we require a distinction between $s$-element subdictionaries. To explore this relation, let us recall the definitions of the RICs, which are more directly related to the subdictionaries of $\A$. We will then show that the mutual coherence can be used to bound the constants involved in the RICs, a fact which will also prove useful in our subsequent analysis. This strategy is inspired by earlier works, which have used the mutual coherence to bound the ERC \cite{tropp06} and the spark \cite{donoho06}. Thus, the coherence can be viewed as a tractable proxy for more accurate measures of the quality of a dictionary, which cannot themselves be calculated efficiently.

By the RICs we refer to two properties describing ``good'' dictionaries, namely, the restricted isometry property (RIP) and the restricted orthogonality property (ROP), which we now define. A dictionary $\A$ is said to satisfy the RIP \cite{candes06} of order $s$ with parameter $\delta_s$ if, for every index set $\Lambda$ of size $s$, we have
\beq \label{eq:RIP}
(1-\delta_s) \|\y\|_2^2 \le \|\AL\y\|_2^2 \le (1+\delta_s)\|\y\|_2^2
\eeq
for all $\y \in \RR^s$. Thus, when $\delta_s$ is small, the RIP ensures that any $s$-atom subdictionary is nearly orthogonal, which in turn implies that any two disjoint $(s/2)$-atom subdictionaries are well-separated.

Similarly, $\A$ is said to satisfy the ROP \cite{candes07} of order $(s_1, s_2)$ with parameter $\theta_{s_1,s_2}$ if, for every pair of disjoint index sets $\Lambda_1$ and $\Lambda_2$ having cardinalities $s_1$ and $s_2$, respectively, we have
\beq \label{eq:ROP}
\left| \y_1^T \A_{\Lambda_1}^T \A_{\Lambda_2} \y_2 \right|
\le \theta_{s_1,s_2} \|\y_1\|_2 \|\y_2\|_2
\eeq
for all $\y_1 \in \RR^{s_1}$ and for all $\y_2 \in \RR^{s_2}$. In words, the ROP requires any two disjoint subdictionaries containing $s_1$ and $s_2$ elements, respectively, to be nearly orthogonal to each other. These two properties are therefore closely related to the requirement that distinct subdictionaries of $\A$ behave dissimilarly.

In recent years, it has been demonstrated that various practical estimation techniques successfully approximate $\x_0$ from $\b$, if the constants $\delta_s$ and $\theta_{s_1,s_2}$ are sufficiently small \cite{candes06, candes07, candes08}. This occurs, for example, when the entries in $\A$ are chosen randomly according to an independent, identically distributed Gaussian law, as well as in some specific deterministic dictionary constructions.

Unfortunately, in the standard estimation setting, one cannot design the system matrix $\A$ according to these specific rules. In general, if one is given a particular dictionary $\A$, then there is no known algorithm for efficiently determining its RICs. Indeed, the very nature of the RICs seems to require enumerating over an exponential number of index sets in order to find the ``worst'' subdictionary. While the mutual coherence $\mu$ of \eqref{eq:def mu} tends to be far less accurate in capturing the accuracy of a dictionary, it is still useful to be able to say something about the RICs based only on $\mu$. Such a result is given in the following lemma.

\begin{lemma} \label{le:RIP ROP}
For any matrix $\A$, the RIP constant $\delta_s$ of \eqref{eq:RIP} and the ROP constant $\theta_{s_1,s_2}$ of \eqref{eq:ROP} satisfy the bounds
\begin{align}
\delta_s         &\le (s-1)\mu,             \label{eq:RIP bound}\\
\theta_{s_1,s_2} &\le \mu \sqrt{s_1 s_2}    \label{eq:ROP bound}
\end{align}
where $\mu$ is the mutual coherence \eqref{eq:def mu}.
\end{lemma}

The proof of Lemma~\ref{le:RIP ROP} can be found in Appendix~\ref{ap:RIP ROP}. We will apply this lemma in Section~\ref{se:random}, when examining the performance of the Dantzig selector. This tool can also be used in conjunction with other results that rely on the RIP and ROP\@.

\subsection{Estimation Techniques}
\label{ss:techniques}

To fix notation, we now briefly review several approaches for estimating $\x_0$ from noisy measurements $\b$ given by \eqref{eq:b=Ax+w}. The two main strategies for efficiently estimating a sparse vector are $\ell_1$ relaxation and greedy methods. The first of these involves solving an optimization problem wherein the nonconvex constraint $\|\x_0\|_0 = s$ is relaxed to a constraint on the $\ell_1$ norm of the estimated vector $\x_0$. Specifically, we consider the $\ell_1$-penalty version of BPDN, which estimates $\x_0$ as a solution $\xbp$ to the quadratic program
\beq \label{eq:bpdn}
\min_\x \tfrac{1}{2} \|\b - \A\x\|_2^2 + \gamma \|\x\|_1
\eeq
for some regularization parameter $\gamma$. We refer to the optimization problem \eqref{eq:bpdn} as BPDN, although it should be noted that some authors reserve this term for the related optimization problem
\beq \label{eq:bpdn l1error}
\min_\x \|\x\|_1 \quad \st \|\b-\A\x\|_2^2 \le \delta
\eeq
where $\delta$ is a given constant.

Another estimator based on the idea of $\ell_1$ relaxation is the Dantzig selector \cite{candes07}, defined as a solution $\xds$ to the optimization problem
\beq \label{eq:ds}
\min_\x \|\x\|_1 \quad \st \|\A^T (\b - \A\x) \|_\infty \le \tau
\eeq
where $\tau$ is again a user-selected parameter. The Dantzig selector, like BPDN, is a convex relaxation method, but rather than penalizing the $\ell_2$ norm of the residual $\b - \A\x$, the Dantzig selector ensures that the residual is weakly correlated with all dictionary atoms.

Instead of solving an optimization problem, greedy approaches estimate the support set $\Lz$ from the measurements $\b$. Once a support set $\Lambda$ is chosen, the parameter vector $\x_0$ can be estimated using least-squares (LS) to obtain
\beq \label{eq:lsest}
\hx =
\begin{cases}
\AL^\pinv \b & \text{on the support set $\Lambda$}, \\
\zero        & \text{elsewhere.}
\end{cases}
\eeq

Greedy techniques differ in the method by which the support set is selected. The simplest method is known as the thresholding algorithm. This technique computes the correlation of the measured signal $\b$ with each of the atoms $\a_i$ and defines $\Lambda$ as the set of indices of the $s$ atoms having the highest correlation. Subsequently, the LS technique \eqref{eq:lsest} is applied to obtain the thresholding estimate $\xth$.

A somewhat more sophisticated greedy algorithm is OMP \cite{pati93}. This iterative approach begins by initializing the estimated support set $\Lambda^0$ to the empty set and setting a residual vector $\r^0$ to $\b$. Subsequently, at each iteration $i = 1,\ldots,s$, the algorithm finds the single atom which is most highly correlated with $\r^{i-1}$. The index of this atom, say $k_i$, is added to the support set, so that $\Lambda^i = \Lambda^{i-1} \cup \{k_i\}$. The estimate $\xomp^i$ at the $i$th iteration is then defined by the LS solution \eqref{eq:lsest} using the support set $\Lambda^i$. Next, the residual is updated using the formula
\beq
\r^i = \b - \A\xomp^i.
\eeq
The residual thus describes the part of $\b$ which has yet to be accounted for by the estimate. The counter $i$ is now incremented, and $s$ iterations are performed, after which the OMP estimate $\xomp$ is defined as the estimate at the final iteration, $\xomp^s$. A well-known property of OMP is that the algorithm never chooses the same atom twice \cite{donoho06}. Consequently, stopping after $s$ iterations guarantees that $\|\xomp\|_0 = s$.

Finally, we also mention the so-called oracle estimator, which is based both on $\b$ and on the true support set $\Lz$ of $\x_0$; the support set is assumed to have been provided by an ``oracle''. The oracle estimator $\xo$ calculates the LS solution \eqref{eq:lsest} for $\Lz$. In the case of white Gaussian noise, the MSE obtained using this technique equals that of the CRB \cite{ben-haim09b}. Thus, it makes sense to use the oracle estimator as a gold standard against which the performance of practical algorithms can be compared.

\section{Performance under Adversarial Noise}
\label{se:adversarial}

In this section, we briefly discuss the case in which the noise $\w$ is an unknown deterministic vector which satisfies $\|\w\|_2 \le \eps$. As we will see, performance guarantees in this case are rather weak, and indeed no denoising capability can be ensured for any known algorithm. In Section~\ref{se:random}, we will compare this setting with the results which can be obtained when $\w$ is random.

Typical ``stability'' results under adversarial noise guarantee that if the mutual coherence $\mu$ of $\A$ is sufficiently small, and if $\x_0$ is sufficiently sparse, then the distance between $\x_0$ and its estimate is on the order of the noise magnitude. Such results can be derived for algorithms including BPDN, OMP, and thresholding. Consider, for example, the following theorem, which is based on the work of Tropp \cite[\S IV-C]{tropp06}.\footnote{Tropp considers only the case in which the entries of $\x_0$ belong to the set $\{ 0, \pm 1 \}$. However, since the analysis performed in \cite[\S IV-C]{tropp06} can readily be applied to the general setting considered here, we omit the proof of Theorem~\ref{th:adversarial}.}

\begin{theorem}[Tropp] \label{th:adversarial}
Let $\x_0$ be an unknown deterministic vector with known sparsity $\|\x_0\|_0 = s$, and let $\b = \A\x_0 + \w$, where $\|\w\|_2 \le \eps$. Suppose the mutual coherence $\mu$ of the dictionary $\A$ satisfies $s < 1/(3\mu)$. Let $\xbp$ denote a solution of BPDN \eqref{eq:bpdn} with regularization parameter $\gamma = 2\eps$. Then, $\xbp$ is unique, the support of $\xbp$ is a subset of the support of $\x_0$, and
\beq
\|\x_0 - \xbp\|_\infty < \left( 3 + \sqrt{\tfrac{3}{2}} \right)\eps \approx 4.22\eps.
\eeq
\end{theorem}

Results similar to Theorem~\ref{th:adversarial} have also been obtained \cite{donoho06, candes06, candes08, cai09} for the related $\ell_1$-error estimation approach \eqref{eq:bpdn l1error}, as well as for the OMP algorithm \cite{donoho06}. Furthermore, the technique used in the proof for the OMP \cite{donoho06} can also be applied to demonstrate a (slightly weaker) performance guarantee for the thresholding algorithm.

In all of the aforementioned results, the only guarantee is that the distance between $\xbp$ and $\x_0$ is on the order of the noise power $\eps$. Such results are somewhat disappointing, because one would expect the knowledge that $\x_0$ is sparse to assist in denoising; yet Theorem~\ref{th:adversarial} promises only that the $\ell_\infty$ distance between $\xbp$ and $\x_0$ is less than about four times the maximum noise level. However, the fact that no denoising has occurred is a consequence of the problem setting itself, rather than a limitation of the algorithms proposed above. In the adversarial case, even the oracle estimator can only guarantee an estimation error on the order of $\eps$. This is because $\w$ can be chosen so that $\w \in \spn(\ALz)$, in which case projection onto $\spn(\ALz)$, as performed by the oracle estimator, does not remove any portion of the noise.

In conclusion, results in this adversarial context must take into account values of $\w$ which are chosen so as to cause maximal damage to the estimation algorithm. In many practical situations, such a scenario is overly pessimistic. Thus, it is interesting to ask what guarantees can be made about the performance of practical estimators under the assumption of random (and thus non-adversarial) noise. This scenario is considered in the next section.

\section{Performance under Random Noise}
\label{se:random}

We now turn to the setting in which the noise $\w$ is a Gaussian random vector with mean $\zero$ and covariance $\sigma^2 \I$. In this case, it can be shown \cite{ben-haim09b} that the MSE of any unbiased estimator of $\x_0$ satisfies the Cram\'er--Rao bound
\beq \label{eq:crb}
\MSE(\hx) \ge \CRB = \sigma^2 \Tr((\ALz^T\ALz)^{-1})
\eeq
whenever $\|\x_0\|_0 = s$. Interestingly, $\CRB$ is also the MSE of the oracle estimator \cite{candes07}.

It follows from the Gershgorin disc theorem \cite{golub96} that all eigenvalues of $\ALz^T\ALz$ are between $1 - (s-1)\mu$ and $1 + (s+1)\mu$. Therefore, for reasonable sparsity levels, $\Tr((\ALz^T\ALz)^{-1})$ is not much larger than $s$; for example, if we assume, as in Theorem~\ref{th:adversarial}, that $s < 1/(3\mu)$, then $\CRB$ of \eqref{eq:crb} is no larger than $\frac 3 2 s \sigma^2$. Considering that the mean power of $\w$ is $n \sigma^2$, it is evident that the oracle estimator has substantially reduced the noise level. In this section, we will demonstrate that comparable performance gains are achievable using practical methods, which do not have access to the oracle.

\subsection{$\ell_1$-Relaxation Approaches}
\label{ss:l1}

Historically, performance guarantees under random noise were first obtained for the Dantzig selector \eqref{eq:ds}. The result, due to Cand\`es and Tao \cite{candes07}, is derived using the RICs \eqref{eq:RIP}--\eqref{eq:ROP}. Using the bounds of Lemma~\ref{le:RIP ROP} yields the following coherence-based result.

\begin{theorem}[Cand\`es and Tao]
\label{th:dantzig}
Let $\x_0$ be an unknown deterministic vector such that $\|\x_0\|_0 = s$, and let $\b = \A\x_0 + \w$, where $\w \sim N(\zero, \sigma^2 \I)$ is a random noise vector. Assume that
\beq \label{eq:dantzig cond}
s < 1 + \frac{1}{(1+\sqrt 2) \mu}
\eeq
and consider the Dantzig selector \eqref{eq:ds} with parameter
\beq \label{eq:dantzig thresh}
\tau = \sigma \sqrt{2 (1+\alpha) \log m}
\eeq
for some constant $\alpha>0$. Then, with probability exceeding
\beq \label{eq:dantzig prob}
1 - \frac{1}{m^\alpha \sqrt{\pi \log m}},
\eeq
the Dantzig selector $\xds$ satisfies
\beq \label{eq:ds err}
\|\x_0 - \xds\|_2^2 \le 2 c_1^2 (1+\alpha) s \sigma^2 \log m
\eeq
where
\beq \label{eq:c1}
c_1 = \frac{4}{1 - \left( ( 1+\sqrt 2 ) s - 1 \right) \mu}.
\eeq
\end{theorem}

This theorem is significant because it demonstrates that, while $\xds$ does not quite reach the performance of the oracle estimator, it does come within a constant factor multiplied by $\log m$, with high probability. Interestingly, the $\log m$ factor is an unavoidable result of the fact that the locations of the nonzero elements in $\x_0$ are unknown (see \cite[\S7.4]{candes06b} and the references therein). 

It is clearly of interest to determine whether results similar to Theorem~\ref{th:dantzig} can be obtained for other sparse estimation algorithms \cite{efron07, candes07b}. In this context, Bickel et al.\ \cite{bickel08} have recently shown that, with high probability, BPDN also comes within a factor of $C \log m$ of the oracle performance, for a constant $C$. In fact, their analysis is quite versatile, and simultaneously provides a result for both the Dantzig selector and BPDN\@. However, the constant $C$ obtained in this BPDN guarantee is always larger than $128$, often substantially so; this is considerably weaker than the result of Theorem~\ref{th:dantzig}. Furthermore, while the necessary conditions for the results of Bickel et al.\ are not directly comparable with those of Cand\`es and Tao, an application of Lemma~\ref{le:RIP ROP} indicates that coherence-based conditions stronger than \eqref{eq:dantzig cond} are required for the results of Bickel et al.\ to hold. 

In the following, we obtain a coherence-based performance guarantee for BPDN\@. In particular, we demonstrate that, for an appropriate choice of the regularization parameter $\gamma$, the squared error of the BPDN estimate is bounded, with high probability, by a small constant times $s \sigma^2 \log(m-s)$, and that this constant is lower than that of Theorem~\ref{th:dantzig}. We begin by stating the following somewhat more general result, whose proof is found in Appendix~\ref{ap:main}.

\begin{theorem}
\label{th:main}
Let $\x_0$ be an unknown deterministic vector with known sparsity $\|\x_0\|_0 = s$, and let $\b = \A\x_0 + \w$, where $\w \sim N(\zero, \sigma^2\I)$ is a random noise vector. Suppose that\footnote{As in \cite{tropp06}, analogous findings can also be obtained under the weaker requirement $s < 1/(2\mu)$, but the resulting expressions are somewhat more involved.}
\beq \label{eq:th cond}
s < \frac{1}{3\mu}.
\eeq
Then, with probability exceeding
\beq \label{eq:th prob}
\left( 1 - (m-s) \exp\left(-\frac{\gamma^2}{8\sigma^2}\right) \right)
\left( 1 - e^{-s/7} \right),
\eeq
the solution $\xbp$ of BPDN \eqref{eq:bpdn} is unique and satisfies
\beq \label{eq:th err}
\|\x_0 - \xbp\|_2^2 \le \left( \sigma \sqrt{3} + \tfrac{3}{2}\gamma \right)^2 s.
\eeq
\end{theorem}

To compare the results for BPDN and the Dantzig selector, we now derive from Theorem~\ref{th:main} a result which holds with a probability on the order of \eqref{eq:dantzig prob}. Observe that in order for \eqref{eq:th prob} to be a high probability, we require $\exp(-\gamma^2/(8\sigma^2))$ to be substantially smaller than $1/(m-s)$. This requirement can be used to select a value for the regularization parameter $\gamma$. In particular, one requires $\gamma$ to be at least on the order of $\sqrt{8 \sigma^2 \log(m-s)}$. However, $\gamma$ should not be much larger than this value, as this will increase the error bound \eqref{eq:th err}. We propose to use
\beq \label{eq:gamma best}
\gamma = \sqrt{8 \sigma^2 (1+\alpha) \log(m-s)}
\eeq
for some fairly small $\alpha>0$. Substituting this value of $\gamma$ into Theorem~\ref{th:main} yields the following result.

\begin{corollary} \label{co:1}
Under the conditions of Theorem~\ref{th:main}, let $\xbp$ be a solution of BPDN \eqref{eq:bpdn} with $\gamma$ given by \eqref{eq:gamma best}. Then, with probability exceeding
\beq \label{eq:co prob}
\left( 1 - \frac{1}{(m-s)^\alpha} \right) \left( 1-e^{-s/7} \right)
\eeq
we have
\beq \label{eq:co err}
\|\x_0 - \xbp\|_2^2 \le
\left( \sqrt{3} + 3\sqrt{2(1+\alpha)\log(m-s)} \right)^2 s\sigma^2.
\eeq
\end{corollary}

Let us examine the probability \eqref{eq:co prob} with which Corollary~\ref{co:1} holds, to verify that it is indeed roughly equal to \eqref{eq:dantzig prob}. The expression \eqref{eq:co prob} consists of a product of two terms, both of which converge to $1$ as the problem dimensions increase. The right-hand term may seem odd because it appears to favor non-sparse signals; however, this is an artifact of the method of proof, which requires a sufficient number of nonzero coefficients for large number approximations to hold. This right-hand term converges to $1$ exponentially and therefore typically has a negligible effect on the overall probability of success; for example, for $s \ge 50$ this term is larger than $0.999$.

The left-hand term in \eqref{eq:co prob} tends to $1$ polynomially as $m-s$ increases. This is a slightly lower rate than the probability \eqref{eq:dantzig prob} with which the Dantzig selector bound holds; however, this difference is compensated for by a correspondingly lower multiplicative factor of $\log(m-s)$ in the BPDN error bound \eqref{eq:co err}, as opposed to the $\log m$ factor in the Dantzig selector. In any case, for both theorems to hold, $m$ must increase much more quickly than $s$, so that these differences are negligible.

For large $s$ and $m-s$, Corollary~\ref{co:1} ensures that, with high probability, $\|\xbp-\x_0\|_2^2$ is no larger than a constant multiplied by $s\sigma^2 \log(m-s)$. Up to a multiplicative constant, this error bound is essentially identical to the result \eqref{eq:ds err} for the Dantzig selector. As we have seen, the probabilities with which these bounds hold are likewise almost identical. However, the constants involved in the BPDN, as demonstrated by Corollary~\ref{co:1}, are substantially lower than those previously known for the Dantzig selector. To see this, consider a situation in which $s = 1/(4\mu)$. In this case, for large $s$, the bound \eqref{eq:ds err} on the Dantzig selector rapidly converges to
\beq \label{eq:dantzig cons}
\|\x_0 - \xds\|_2^2 \le 203.6 (1+\alpha) \cdot \log m \cdot s \sigma^2.
\eeq
By comparison, the performance of BPDN in the same setting, as bounded by Corollary~\ref{co:1}, is
\beq \label{eq:co1 cons}
\|\x_0 - \xbp\|_2^2 \le 18 (1+\alpha) \cdot \log(m-s) \cdot s \sigma^2
\eeq
which is over 10 times lower. This improvement is not merely a result of the particular choice of $s$ or $\mu$. Indeed, the multiplicative factor of $18$ which appeared in the BPDN bound \eqref{eq:co1 cons} holds for large $s$ with any value of $\mu$, as long as $s < 1/(3\mu)$; whereas it can be seen from \eqref{eq:ds err}--\eqref{eq:c1} that the multiplicative factor of the Dantzig selector is always larger than $32$. Further comparison between these guarantees will be presented in Section~\ref{se:numer}.

\subsection{Greedy Approaches}

The performance guarantees obtained for the $\ell_1$-relaxation techniques required only the assumption that $\x_0$ is sufficiently sparse. By contrast, for greedy algorithms, successful estimation can only be guaranteed if one further assumes that all nonzero components of $\x_0$ are somewhat larger than the noise level. The reason is that greedy techniques are based on a LS solution for an estimated support, an approach whose efficacy is poor unless the support is correctly identified. Indeed, when using the LS technique \eqref{eq:lsest}, even a single incorrectly identified support element may cause the entire estimate to be severely incorrect. To ensure support recovery, all nonzero elements must be large enough to overcome the noise.

To formalize this notion, denote $\x_0 = (x_{0,1}, \ldots, x_{0,m})^T$ and define
\begin{align} \label{eq:def xmin xmax}
\xmin &= \min_{i \in \Lz} |x_{0,i}|, \notag\\
\xmax &= \max_{i \in \Lz} |x_{0,i}|.
\end{align}
A performance guarantee for both OMP and the thresholding algorithm is then given by the following theorem, whose proof can be found in Appendix~\ref{ap:omp}.

\begin{theorem} \label{th:omp}
Let $\x_0$ be an unknown deterministic vector with known sparsity $\|\x_0\|_0 = s$, and let $\b = \A\x_0 + \w$, where $\w \sim N(\zero, \sigma^2\I)$ is a random noise vector. Suppose that
\beq \label{eq:omp given}
\xmin - (2s-1)\mu\xmin \ge 2\sigma \sqrt{2(1+\alpha)\log m}
\eeq
for some constant $\alpha > 0$. Then, with probability at least
\beq \label{eq:omp prob}
1 - \frac{1}{m^\alpha \sqrt{\pi (1+\alpha) \log m}} ,
\eeq
the OMP estimate $\xomp$ identifies the correct support $\Lz$ of $\x_0$ and, furthermore, satisfies
\begin{subequations}\label{eq:omp err}
\begin{align}
\left\|\xomp - \x_0\right\|_2^2
&\le \frac{2(1+\alpha)}{(1-(s-1)\mu)^2} s \sigma^2 \log m  \label{eq:omp err strong} \\
&\le 8 (1+\alpha) s \sigma^2 \log m. \label{eq:omp err simpl}
\end{align}
\end{subequations}
If the stronger condition
\beq \label{eq:th given}
\xmin - (2s-1)\mu\xmax \ge 2\sigma \sqrt{2(1+\alpha)\log m}
\eeq
holds, then with probability exceeding \eqref{eq:omp prob}, the thresholding algorithm also correctly identifies $\Lz$ and satisfies \eqref{eq:omp err}.
\end{theorem}

The performance guarantee \eqref{eq:omp err} is better than that provided by Theorem~\ref{th:dantzig} and~Corollary~\ref{co:1}. However, this result comes at the expense of requirements on the magnitude of the entries of $\x_0$. Our analysis thus suggests that greedy approaches may outperform $\ell_1$-based methods when the entries of $\x_0$ are large compared with the noise, but that the greedy approaches will deteriorate when the noise level increases. As we will see in Section~\ref{se:numer}, simulations also appear to support this conclusion.

It is interesting to compare the success conditions \eqref{eq:omp given} and \eqref{eq:th given} of the OMP and thresholding algorithms. For given problem dimensions, the OMP algorithm requires $\xmin$, the smallest nonzero element of $\x_0$, to be larger than a constant multiple of the noise standard deviation $\sigma$. This is required in order to ensure that all elements of the support of $\x_0$ will be identified with high probability. The requirement of the thresholding algorithm is stronger, as befits a simpler approach: In this case $\xmin$ must be larger than the noise standard deviation plus a constant times $\xmax$. In other words, one must be able to separate $\xmin$ from the combined effect of noise and interference caused by the other nonzero components of $\x_0$. This results from the thresholding technique, in which the entire support is identified simultaneously from the measurements. By comparison, the iterative approach used by OMP identifies and removes the large elements in $\x_0$ first, thus facilitating the identification of the smaller elements in later iterations.

\section{Numerical Results}
\label{se:numer}

\begin{figure*}
\centerline{\subfigure[Dantzig selector]{\includegraphics{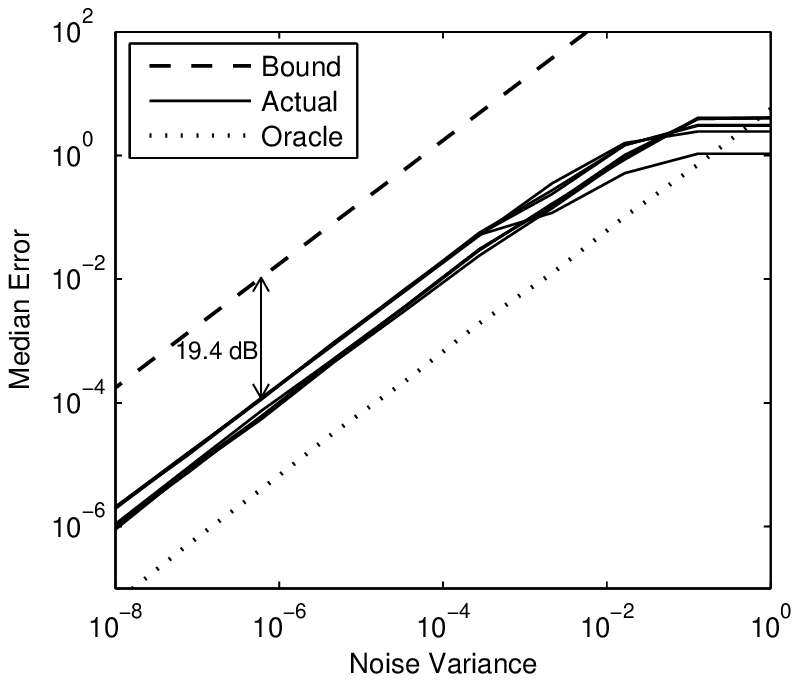}}
\hfil
\subfigure[BPDN]{\includegraphics{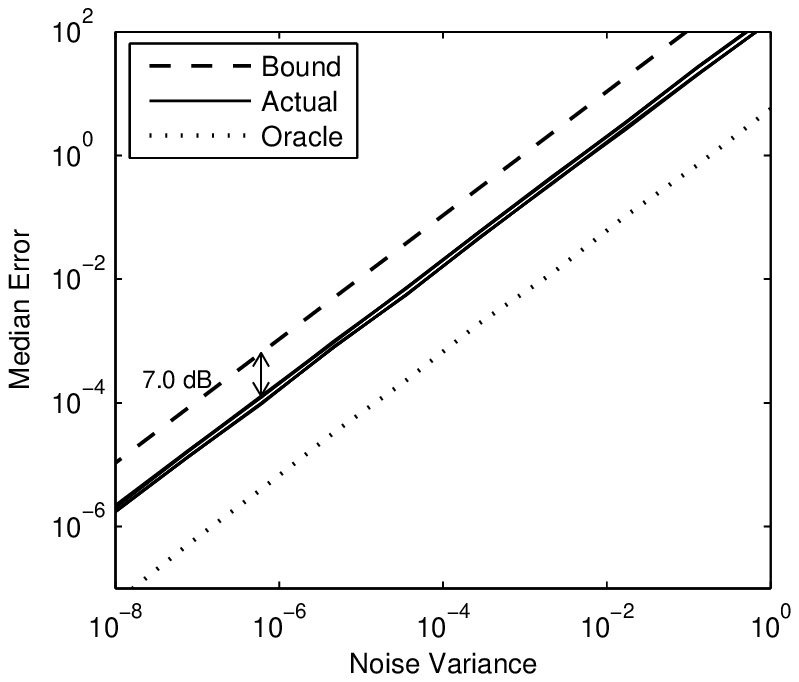}}}
\centerline{\subfigure[OMP]{\includegraphics{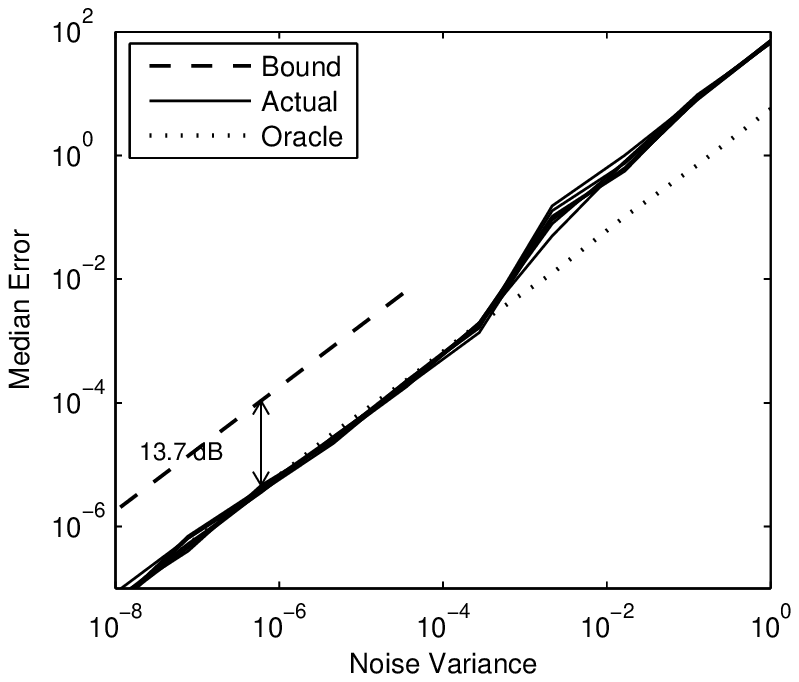}}
\hfil
\subfigure[Thresholding]{\includegraphics{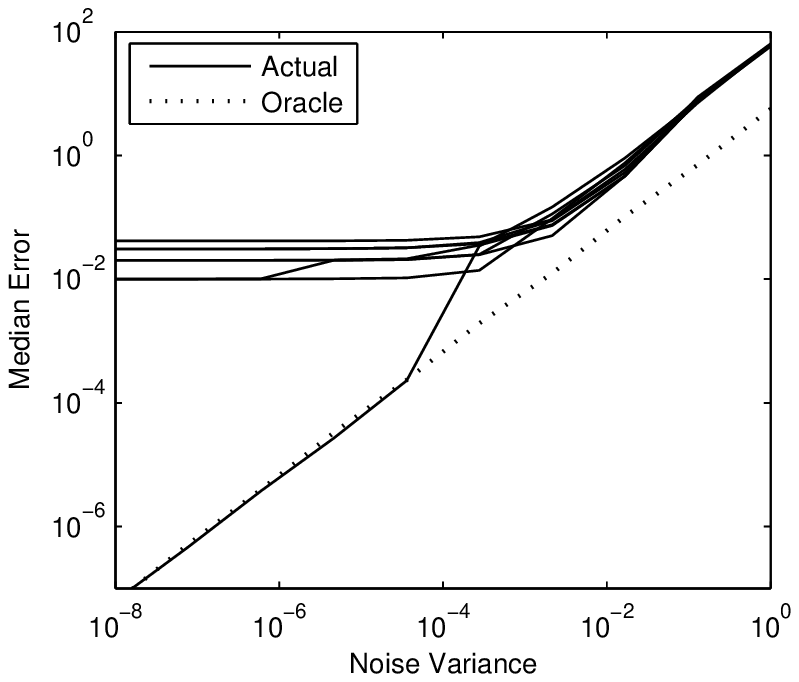}}}
\caption{Median estimation error for practical estimators (solid line) compared with the performance guarantees (dashed line) and the oracle estimator (dotted line). The solid lines report performance for $8$ different values of the unknown parameter vector $\x_0$. For OMP, performance is only guaranteed for $\sigma \le 0.057$, while for thresholding, nothing can be guaranteed for the given problem dimensions.}
\label{fi:median}
\end{figure*}

In this section, we describe a number of numerical experiments comparing the performance of various estimators to the guarantees of Section~\ref{se:random}. Our first experiment measured the median estimation error, i.e., the median of the $\ell_2$ distance between $\x_0$ and its estimate. The median error is intuitively appealing as it characterizes the ``typical'' estimation error, and it can be readily bounded by the performance guarantees of Section~\ref{se:random}.

Specifically, we chose the two-ortho dictionary $\A = [\I \ \ \H]$, where $\I$ is the $512 \times 512$ identity matrix and $\H$ is the $512 \times 512$ Hadamard matrix with normalized columns. The RICs of this dictionary are unknown, but the coherence can be readily calculated and is given by $\mu = 1/\sqrt{512}$. Consequently, the theorems of Section~\ref{se:random} can be used to obtain performance guarantees for sufficiently sparse vectors. In particular, in our simulations we chose parameters $\x_0$ having a support of size $s=7$. The smallest nonzero entry in $\x_0$ was $\xmin = 0.1$ and the largest entry was $\xmax = 1$. Under these conditions, applying the theorems of Section~\ref{se:random} yields the bounds\footnote{In the current setting, the results for the Dantzig selector (Theorem~\ref{th:dantzig}) and OMP (Theorem~\ref{th:omp}) can only be used to yield guarantees holding with probabilities of approximately $3/4$ and higher. These are, of course, also bounds on the median error.}
\begin{align} \label{eq:guar}
\|\x_0 - \xomp\|_2^2 \le   3.7 &s \sigma^2 \log m & \text{w.p. }&\tfrac 3 4 ,  \text{ if } \sigma \le 0.057; \notag\\
\|\x_0 - \xbp\|_2^2  \le  22.1 &s \sigma^2 \log m & \text{w.p. }&\tfrac 1 2 ; \notag\\
\|\x_0 - \xds\|_2^2  \le 361.8 &s \sigma^2 \log m & \text{w.p. }&\tfrac 3 4 .
\end{align}
We have thus obtained guarantees for the median estimation error of the Dantzig selector, BPDN, and OMP\@. Under these settings, no guarantee can be made for the performance of the thresholding algorithm. Indeed, as we will see, for some choices of $\x_0$ satisfying the above requirements, the performance of the thresholding algorithm is not proportional to $s \sigma^2 \log m$. To obtain thresholding guarantees, one requires a narrower range between $\xmin$ and $\xmax$.

To measure the actual median error obtained by various estimators, $8$ different parameter vectors $\x_0$ were selected. These differed in the distribution of the magnitudes of the nonzero components within the range $[\xmin, \xmax]$ and in the locations of the nonzero elements. For each parameter $\x_0$, a set of measurement vectors $\b$ were obtained from \eqref{eq:b=Ax+w} by adding white Gaussian noise. The estimation algorithms of Section~\ref{ss:techniques} were then applied to each measurement realization; for the Dantzig selector and BPDN, the parameters $\tau$ and $\gamma$ were chosen as the smallest values such that the probabilities of success \eqref{eq:dantzig prob} and \eqref{eq:co prob}, respectively, would exceed $1/2$. The median over noise realizations of the distance $\|\x_0-\hx\|_2^2$ was then computed for each estimator. This process was repeated for $10$ values of the noise variance $\sigma^2$ in the range $10^{-8} \le \sigma^2 \le 1$. The results are plotted in Fig.~\ref{fi:median} as a function of $\sigma^2$.

It is evident from Fig.~\ref{fi:median} that some parameter vectors are more difficult to estimate than others. Indeed, there is a large variety of parameters $\x_0$ satisfying the problem requirements, and it is likely that some of them come closer to the theoretical limits than the parameters chosen in our experiment. This highlights the importance of performance guarantees in ensuring adequate performance for \emph{all} parameter values. On the other hand, it is quite possible that further improvements of the constants in the performance bounds are possible. For example, the Dantzig selector guarantee, which is obtained by applying coherence bounds to RIC-based results \cite{candes07}, is almost $100$ times higher than the worst of the examined parameter values. It should also be noted that applying coherence bounds to RIC-based BPDN guarantees \cite{bickel08} yields a bound which applies to the aforementioned matrix $\A$ only when $s \le 3$, and thus cannot be used in the present setting. Therefore, it appears that when dealing with dictionaries for which only the coherence $\mu$ is known, guarantees based directly on $\mu$ are tighter than RIC-based results.

In practice, it is more common to measure the MSE of an estimator than its median error. Our next goal is to determine whether the behavior predicted by our theoretical analysis is also manifested in the MSE of the various estimators. To this end, we conducted an experiment in which the MSEs of the estimators of Section~\ref{ss:techniques} were compared. In this simulation, we chose the two-ortho dictionary $\A = [\I \ \ \H]$, where $\I$ is the $256 \times 256$ identity matrix and $\H$ is the $256 \times 256$ Hadamard matrix with normalized columns.\footnote{Similar experiments were performed on a variety of other dictionaries, including an overcomplete DCT \cite{elad06} and a matrix containing Gaussian random entries. The different dictionaries yielded comparable results, which are not reported here.} Once again, the RICs of this dictionary are unknown. However, the coherence in this case is given by $\mu = 1/16$, and consequently, the $\ell_1$ relaxation guarantees of Section~\ref{ss:l1} hold for $s \le 5$.

\begin{figure}
\centerline{\includegraphics{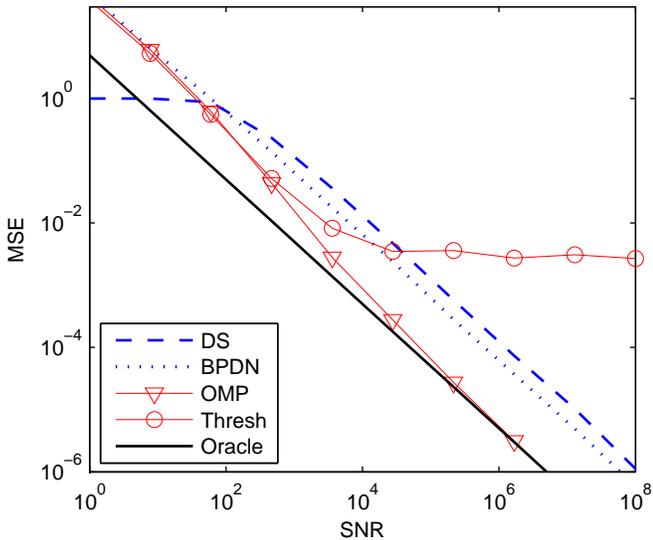}}
\caption{MSE of various estimators as a function of the SNR\@. The sparsity level is $s=5$ and the dictionary is a $256 \times 512$ two-ortho matrix.}
\label{fi:sim1}
\end{figure}

We obtained the parameter vector $\x_0$ for this experiment by selecting a 5-element support at random, choosing the nonzero entries from a white Gaussian distribution, and then normalizing the resulting vector so that $\|\x_0\|_2 = 1$. The regularization parameters $\tau$ and $\gamma$ of the Dantzig selector and BPDN were chosen as recommended by Theorem~\ref{th:dantzig} and Corollary~\ref{co:1}, respectively; for both estimators a value of $\alpha=1$ was chosen, so that the guaranteed probability of success for the two algorithms has the same order of magnitude. The MSE of each estimate was then calculated by averaging over repeated realizations of $\x_0$ and the noise. The experiment was conducted for 10 values of the noise variance $\sigma^2$ and the results are plotted in Fig.~\ref{fi:sim1} as a function of the signal-to-noise ratio (SNR), which is defined by
\beq \label{eq:def snr}
\text{SNR} = \frac{\|\x_0\|_2^2}{n \sigma^2} = \frac{1}{n \sigma^2}.
\eeq

To compare this plot with the theoretical results of Section~\ref{se:random}, observe first the situation at high SNR\@. In this case, OMP, BPDN, and the Dantzig selector all achieve performance which is proportional to the oracle MSE (or CRB) given by \eqref{eq:crb}. Among these, OMP is closest to the CRB, followed by BPDN and, finally, the Dantzig selector. This behavior matches the proportionality constants given in the theorems of Section~\ref{se:random}. Indeed, for small $\sigma$, the condition \eqref{eq:omp given} holds even for large $\alpha$, and thus Theorem~\ref{th:omp} guarantees that OMP will recover the correct support of $\x_0$ with high probability, explaining the convergence of this estimator to the oracle. By contrast, the performance of the thresholding algorithm levels off at high SNR; this is again predicted by Theorem~\ref{th:omp}, since, even when $\sigma=0$, the condition \eqref{eq:th given} does not always hold, unless $\xmin$ is not much smaller than $\xmax$. Thus, for our choice of $\x_0$, Theorem~\ref{th:omp} does not guarantee near-oracle performance for the thresholding algorithm, even at high SNR\@.

With increasing noise, Theorem~\ref{th:omp} requires a corresponding increase in $\xmin$ to guarantee the success of the greedy algorithms. Consequently, Fig.~\ref{fi:sim1} demonstrates a deterioration of these algorithms when the SNR is low. On the other hand, the theorems for the relaxation algorithms make no such assumptions, and indeed these approaches continue to perform well, compared with the oracle estimator, even when the noise level is high. In particular, the Dantzig selector outperforms the CRB at low SNR; this is because the CRB is a bound on unbiased techniques, whereas when the noise is large, biased techniques such as an $\ell_1$ penalty become very effective. Robustness to noise is thus an important advantage of $\ell_1$-relaxation techniques.

\begin{figure}
\centerline{\includegraphics{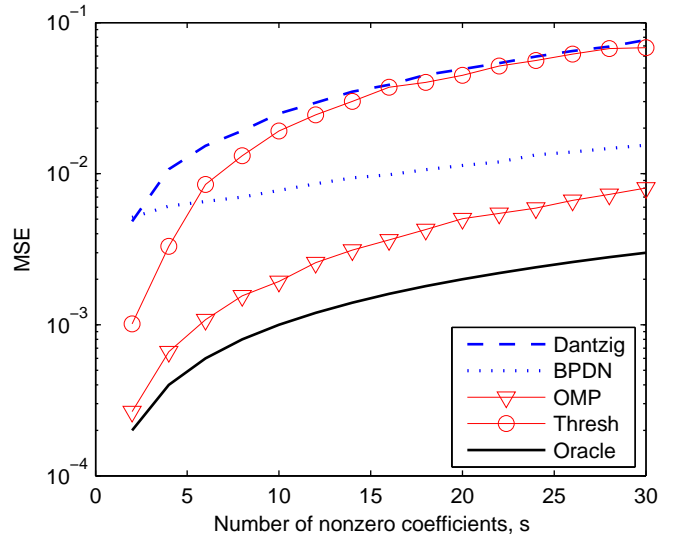}}
\caption{MSE of various estimators as a function of the support size $s$. The noise standard deviation is $\sigma=0.01$ and the dictionary is a $256 \times 512$ two-ortho matrix.}
\label{fi:sim}
\end{figure}

It is also interesting to examine the effect of the support size $s$ on the performance of the various estimators. To this end, 15 support sizes in the range $2 \le s \le 30$ were tested. For each value of $s$, random vectors $\x_0$ having $s$ nonzero entries were selected as in the previous simulation. The dictionary $\A$ was the $256 \times 512$ two-ortho matrix defined above; as in the previous experiment, other matrices were also tested and provided similar results. The standard deviation of the noise for this experiment was $\sigma = 0.01$. The results are plotted in Fig.~\ref{fi:sim}.

As mentioned above, the mutual coherence of the dictionary $\A$ is $1/16$, so that the proposed performance guarantees apply only when $\x_0$ is quite sparse ($s \le 5$). Nevertheless, Fig.~\ref{fi:sim} demonstrates that the estimation algorithms (with the exception of the thresholding approach) exhibit a graceful degradation as the support of $\x_0$ increases. At first sight this would appear to mean that the performance guarantees provided are overly pessimistic. For example, it is possible that the RICs in the present setting, while unknown, are fairly low and permit a stronger analysis than that of Section~\ref{se:random}. It is also quite reasonable to expect, as mentioned above, that some improvement in the theoretical guarantees is possible. However, it is worth recalling that the performance guarantees proposed in this paper apply to all sparse vectors, while the numerical results describe the performance averaged over different values of $\x_0$. Thus it is possible that there exist particular parameter values for which the performance is considerably poorer than that reported in Fig.~\ref{fi:sim}. Indeed, there exist values of $\A$ and $\x_0$ for which BPDN yields grossly incorrect results even when $\|\x_0\|_0$ is on the order of $1/\mu$ \cite{CandesPlan09}. However, identifying such worst-case parameters numerically is quite difficult; this is doubtlessly at least part of the reason for the apparent pessimism of the performance guarantees.

\section{Conclusion}

The performance of an estimator depends on the problem setting under consideration. As we have seen, under the adversarial noise scenario of Section~\ref{se:adversarial}, the estimation error of any algorithm can be as high as the noise power; in other words, the assumption of sparsity has not yielded any denoising effect. On the other hand, in the Bayesian regime in which both $\x_0$ and the noise vector are random, practical estimators come close to the performance of the oracle estimator \cite{CandesPlan09}. In Section~\ref{se:random}, we examined a middle ground between these two extremes, namely the setting in which $\x_0$ is deterministic but the noise is random. As we have shown, despite the fact that less is known about $\x_0$ in this case than in the Bayesian scenario, a variety of estimation techniques are still guaranteed to achieve performance close to that of the oracle estimator.

Our theoretical and numerical results suggest some conclusions concerning the choice of an estimator. In particular, at high SNR values, it appears that the greedy OMP algorithm has an advantage over the other algorithms considered herein. In this case the support set of $\x_0$ can be recovered accurately and OMP thus converges to the oracle estimator; by contrast, $\ell_1$ relaxations have a shrinkage effect which causes a loss of accuracy at high SNR\@. This is of particular interest since greedy algorithms are also computationally more efficient than relaxation methods. On the other hand, the $\ell_1$ relaxation techniques, and particularly the Dantzig selector, appear to be more effective than the greedy algorithms when the noise level is significant: in this case, shrinkage is a highly effective denoising technique. Indeed, as a result of the bias introduced by the shrinkage, $\ell_1$-based approaches can even perform better than the oracle estimator and the Cram\'er--Rao bound.

\appendices
\section{Proof of Lemma~\ref{le:RIP ROP}}
\label{ap:RIP ROP}

By Gershgorin's disc theorem \cite{golub96}, all eigenvalues of $\AL^T\AL$ are between $1 - (s-1)\mu$ and $1 + (s-1)\mu$. Combining this with the fact that, for all $\y$,
\beq
\lambda_{\min}(\AL^T\AL) \|\y\|_2^2
\le \|\AL\y\|_2^2
\le \lambda_{\max}(\AL^T\AL) \|\y\|_2^2,
\eeq
we obtain \eqref{eq:RIP bound}. Next, to demonstrate \eqref{eq:ROP bound}, observe that
\begin{align}
\left| \y_1^T \A_{\Lambda_1}^T \A_{\Lambda_2} \y_2 \right|
&\le \left|\y_1^T\right| \cdot \left|\A_{\Lambda_1}^T \A_{\Lambda_2}\right| \cdot \left|\y_2\right|
\end{align}
where the absolute value of a matrix or vector is taken elementwise. Since $\A_{\Lambda_1}^T \A_{\Lambda_2}$ is a submatrix of $\A^T\A$ which does not contain any of the diagonal elements of $\A^T\A$, it follows that each element in $\A_{\Lambda_1}^T \A_{\Lambda_2}$ is smaller in absolute value than $\mu$. Thus
\begin{align}
\left| \y_1^T \A_{\Lambda_1}^T \A_{\Lambda_2} \y_2 \right|
&\le \mu \left|\y_1^T\right| \one \one^T \left|\y_2\right| 
= \mu \|\y_1\|_1 \|\y_2\|_1
\end{align}
where $\one$ indicates a vector of ones. Using the fact that $\|\y\|_1 \le \sqrt{s} \|\y\|_2$ for any $s$-vector $\y$, we obtain
\begin{align}
\left| \y_1^T \A_{\Lambda_1}^T \A_{\Lambda_2} \y_2 \right|
&\le \mu \sqrt{s_1 s_2} \|\y_1\|_2 \|\y_2\|_2,
\end{align}
which implies that $\theta_{s_1,s_2}$ satisfies \eqref{eq:ROP bound}.

\section{Proof of Theorem~\ref{th:main}}
\label{ap:main}

The proof is based closely on the work of Tropp \cite{tropp06}. From the triangle inequality,
\beq \label{eq:thprf1}
\|\x_0 - \xbp\|_2 \le \|\x_0 - \xo\|_2 + \|\xo - \xbp\|_2
\eeq
where $\xo$ is the oracle estimator. Our goal is to separately bound the two terms on the right-hand side of \eqref{eq:thprf1}. Indeed, as we will see, the two constants $\sigma\sqrt{3}$ and $\tfrac{3}{2}\gamma$ in \eqref{eq:th err} arise, respectively, from the two terms in \eqref{eq:thprf1}.

Beginning with the term $\|\x_0 - \xo\|_2$, let $\xzL$ denote the $s$-vector containing the elements of $\x_0$ indexed by $\Lz$, and similarly, let $\xoL$ denote the corresponding subvector of $\xo$. We then have
\begin{align}
   \xzL - \xoL
&= \xzL - \ALz^\pinv (\A\x_0 + \w) \notag\\
&= \xzL - \ALz^\pinv(\ALz\xzL + \w) \notag\\
&= -\ALz^\pinv \w,
\end{align}
where we have used the fact that $\ALz$ has full column rank, which is a consequence \cite{donoho03} of the condition \eqref{eq:th cond}. Thus, $\xzL - \xoL$ is a Gaussian random vector with mean $\zero$ and covariance $\sigma^2 \ALz^\pinv \ALz^{\pinv T} = \sigma^2 (\ALz^T \ALz)^{-1}$.

For future use, we note that the cross-correlation between $\ALz^\pinv\w$ and $(\I-\ALz\ALz^\pinv)\w$ is
\begin{align}
   \E{ \ALz^\pinv \w \w^T (\I - \ALz\ALz^\pinv)^T }
&= \sigma^2 \ALz^\pinv (\I - \ALz\ALz^\pinv)^T \notag\\
&= \zero,
\end{align}
where we have used the fact \cite[Th.~1.2.1]{campbell79} that for any matrix $\M$
\beq
\M^\pinv \M^{\pinv T} \M^T = (\M^T \M)^\pinv \M^T = \M^\pinv.
\eeq
Since $\w$ is Gaussian, it follows that $\ALz^\pinv\w$ and $(\I-\ALz\ALz^\pinv)\w$ are statistically independent. Furthermore, because $\xzL-\xoL$ depends on $\w$ only through $\ALz^\pinv\w$, we conclude that
\beq \label{eq:thprf indep}
\x_0-\xo \text{ is statistically independent of } (\I-\ALz\ALz^\pinv)\w.
\eeq

We now wish to bound the probability that $\|\x_0 - \xo\|_2^2 > 3s\sigma^2$. Let $\z$ be a normalized Gaussian random variable, $\z \sim N(\zero, \I_s)$. Then
\begin{align} \label{eq:thprf event1}
&    \Pr{ \|\x_0 - \xo\|_2^2 > 3 s\sigma^2 } \notag\\
&=   \Pr{ \left\| \sigma (\ALz^T\ALz)^{-1/2} \z \right\|_2^2 \ge 3s\sigma^2 } \notag\\
&\le \Pr{ \left\|(\ALz^T\ALz)^{-1/2}\right\|^2 \|\z\|_2^2 \ge 3s }
\end{align}
where $\|\M\|$ denotes the maximum singular value of the matrix $\M$. Thus, $\|(\ALz^T\ALz)^{-1/2}\| = 1/s_{\min}$, where $s_{\min}$ is the minimum singular value of $\ALz$. From the Gershgorin disc theorem \cite[p.~320]{golub96}, it follows that $s_{\min} \ge \sqrt{1 - (s-1)\mu}$. Using \eqref{eq:th cond}, this can be simplified to $s_{\min} \ge \sqrt{2/3}$, and therefore
\beq
\left\|(\ALz^T\ALz)^{-1/2}\right\| \le \sqrt{\frac{3}{2}}.
\eeq
Combining with \eqref{eq:thprf event1} yields
\beq \label{eq:thprf2}
    \Pr{ \|\x_0 - \xo\|_2^2 > 3 s\sigma^2 }
\le \Pr{ \|\z\|_2^2 \ge 2s}.
\eeq
Observe that $\|\z\|_2^2$ is the sum of $s$ independent normalized Gaussian random variables. The right-hand side of \eqref{eq:thprf2} is therefore $1 - \Fchi(2s)$, where $\Fchi(\cdot)$ is the cumulative distribution function of the $\chi^2$ distribution with $s$ degrees of freedom. Using the formula \cite[\S16.3]{kendall-vol1} for $\Fchi(\cdot)$, we have
\beq
    \Pr{ \|\x_0 - \xo\|_2^2 > 3s\sigma^2 }
\le Q\!\left( \tfrac{1}{2}s, s \right)
\eeq
where $Q(a,z)$ is the regularized Gamma function
\beq
Q(a,z) \triangleq \frac{\int_z^\infty t^{a-1} e^{-t} dt}{\int_0^\infty t^{a-1} e^{-t} dt}.
\eeq
$Q\!\left( \tfrac{1}{2}s, s \right)$ decays exponentially as $s \rightarrow \infty$, and it can be seen that
\beq \label{eq:thprf3}
Q\!\left( \tfrac{1}{2}s, s \right) < e^{-s/7} \quad \text{for all $s$.}
\eeq
We thus conclude that the event
\beq \label{eq:thprf4}
\|\x_0 - \xo\|_2^2 \le 3s\sigma^2
\eeq
occurs with probability no smaller than $1 - e^{-s/7}$. Note that the same technique can be applied to obtain bounds on the probability that $\|\x_0 - \xo\|_2^2 > \alpha s\sigma^2$, for any $\alpha>\frac 2 3$. The only difference will be the rate of exponential decay in \eqref{eq:thprf3}. However, the distance between $\x_0$ and $\xo$ is usually small compared with the distance between $\xo$ and $\xbp$, so that such an approach does not significantly affect the overall result.

The above calculations provided a bound on the first term in \eqref{eq:thprf1}. To address the second term $\|\xo-\xbp\|_2$, define the random event
\beq \label{eq:def G}
G : \ \max_i \left| \a_i^T (\I - \ALz\ALz^\pinv) \b \right| \le \tfrac{1}{2} \gamma
\eeq
where $\a_i$ is the $i$th column of $\A$. It is shown in \cite[App.~IV-A]{tropp06} that
\begin{align}\label{eq:thprfa}
\Pr{ G }
&\ge 1 - (m-s)\exp\left(-\frac{\gamma^2}{8\sigma^2}\right).
\end{align}
If $G$ indeed occurs, then the portion of the measurements $\b$ which do not belong to the range space of $\ALz$ are small, and consequently it has been shown \cite[Cor.~9]{tropp06} that, in this case, the solution $\xbp$ to \eqref{eq:bpdn} is unique, the support of $\xbp$ is a subset of $\Lz$, and
\beq
\|\xbp - \xo\|_\infty \le \tfrac{3}{2} \gamma.
\eeq
Since both $\xbp$ and $\xo$ are nonzero only in $\Lz$, this implies that
\beq \label{eq:thprf b}
\|\xbp - \xo\|_2 \le \tfrac{3}{2} \gamma \sqrt{s}.
\eeq

The event $G$ depends on the random variable $\w$ only through $(\I-\ALz\ALz^\pinv)\w$. Thus, it follows from \eqref{eq:thprf indep} that $G$ is statistically independent of the event \eqref{eq:thprf event1}. The probability that both events occur simultaneously is therefore given by the product of their respective probabilities. In other words, with probability exceeding \eqref{eq:th prob}, both \eqref{eq:thprf b} and \eqref{eq:thprf4} hold. Using \eqref{eq:thprf1} completes the proof of the theorem.

\section{Proof of Theorem~\ref{th:omp}}
\label{ap:omp}

The claims concerning both algorithms are closely related. To emphasize this similarity, we first provide several lemmas which will be used to prove both results. These lemmas are all based on an analysis of the random event
\beq \label{eq:def B}
B = \left\{ \max_{1 \le i \le m} |\a_i^T \w| < \tau \right\}
\eeq
where
\beq \label{eq:def tau}
\tau \triangleq \sigma \sqrt{2 (1+\alpha) \log m}
\eeq
and $\alpha > 0$. Our proof will be based on demonstrating that $B$ occurs with high probability, and that when $B$ does occur, both thresholding and OMP achieve near-oracle performance.

\begin{lemma} \label{le:prob B}
Suppose that $\w \sim N(\zero, \sigma^2 \I)$. Then, the event $B$ of \eqref{eq:def B} occurs with a probability of at least \eqref{eq:omp prob}.
\end{lemma}

\begin{proof}
The random variables $\{\a_i^T \w\}_{i=1}^m$ are jointly Gaussian. Therefore, by \v{S}id\'ak's lemma \cite[Th.~1]{sidak67}
\beq \label{eq:th prf 3}
\Pr{B}
=   \Pr{ \max_{1 \le i \le m} |\a_i^T \w| < \tau }
\ge \prod_{i=1}^m \Pr{|\a_i^T \w| \le \tau}.
\eeq
Since $\|\a_i\|_2 = 1$, each random variable $\a_i^T \w$ has mean zero and variance $\sigma^2$. Consequently,
\beq \label{eq:th prf 2}
\Pr{|\a_i^T \w| < \tau} = 1 - 2 \Qfunc{\frac{\tau}{\sigma}}
\eeq
where $\Qfunc{x} = (1/\sqrt{2\pi}) \int_x^\infty e^{-z^2/2} dz$ is the Gaussian tail probability. Using the bound
\beq
\Qfunc{x} \le \frac{1}{x \sqrt{2\pi}} e^{-x^2/2}
\eeq
we obtain from \eqref{eq:th prf 2}
\beq \label{eq:th eta}
\Pr{|\a_i^T \w| < \tau} \ge 1 - \eta
\eeq
where
\beq
\eta \triangleq \sqrt{\frac{2}{\pi}} \cdot \frac{\sigma}{\tau} e^{-\tau^2/2\sigma^2}.
\eeq
When $\eta > 1$, the bound \eqref{eq:omp prob} is meaningless and the theorem holds vacuously. Otherwise, when $\eta \le 1$, we have from \eqref{eq:th prf 3} and \eqref{eq:th eta}
\beq
\Pr{B} \ge (1 - \eta)^m \ge 1 - m \eta
\eeq
where the final inequality holds for any $\eta \le 1$ and $m \ge 1$. Substituting the values of $\eta$ and $\tau$ and simplifying, we obtain that $B$ holds with a probability no lower than \eqref{eq:omp prob}, as required.
\end{proof}

The next lemma demonstrates that, under suitable conditions, correlating $\b$ with the dictionary atoms $\a_i$ is an effective method of identifying the atoms participating in the support of $\x_0$.

\begin{lemma} \label{le:xmax}
Let $\x_0$ be a vector with support $\Lz = \supp(\x_0)$ of size $s=|\Lz|$, and let $\b = \A\x_0 + \w$ for some noise vector $\w$. Define $\xmin$ and $\xmax$ as in \eqref{eq:def xmin xmax}, and suppose that
\beq \label{eq:xmax cond}
\xmax - (2s-1)\mu \xmax \ge 2\tau.
\eeq
Then, if the event $B$ of \eqref{eq:def B} holds, we have
\beq \label{eq:xmax res}
\max_{j \in \Lz} |\a_j^T \b| > \max_{j \notin \Lz} |\a_j^T \b|.
\eeq
If, rather than \eqref{eq:xmax cond}, the stronger condition
\beq \label{eq:xminmax cond}
\xmin - (2s-1)\mu \xmax \ge 2\tau
\eeq
is given, then, under the event $B$, we have
\beq \label{eq:xminmax res}
\min_{j \in \Lz} |\a_j^T \b| > \max_{j \notin \Lz} |\a_j^T \b|.
\eeq
\end{lemma}

\begin{proof}
The proof is an adaptation of \cite[Lemma~5.2]{donoho06}. Beginning with the term $\max_{j \notin \Lz} | \a_j^T \b |$, we have, under the event $B$,
\begin{align} \label{eq:th prf 1}
    \max_{j \notin \Lz} | \a_j^T \b |
&=  \max_{j \notin \Lz} \left| \a_j^T \w + \sum_{i \in \Lz} x_i \a_j^T \a_i \right| \notag\\
&\le\max_{j \notin \Lz} | \a_j^T \w | + \max_{j \notin \Lz} \sum_{i \in \Lz} \left| x_i \a_j^T \a_i \right| \notag\\
&<  \tau + s \mu \xmax.
\end{align}
On the other hand, when $B$ holds,
\begin{align}
    \max_{j \in \Lz} | \a_j^T \b |
&=  \max_{j \in \Lz} \left| x_j + \a_j^T \w + \sum_{i \in \Lz\backslash\{j\}} x_i \a_j^T \a_i \right| \notag\\
&\ge\xmax - \max_{j \in \Lz} \left| \a_j^T \w + \sum_{i \in \Lz\backslash\{j\}} x_i \a_j^T \a_i \right| \notag\\
&>  \xmax - \tau - (s-1)\mu\xmax \notag\\
&=  \xmax - (2s-1)\mu\xmax -\tau + s\mu\xmax .
\end{align}
Together with \eqref{eq:th prf 1}, this yields
\beq
\max_{j \in \Lz} | \a_j^T \b | >  \xmax - (2s-1)\mu\xmax -2\tau + \max_{j \notin \Lz} | \a_j^T \b |.
\eeq
Thus, under the condition \eqref{eq:xmax cond}, we obtain \eqref{eq:xmax res}. Similarly, when $B$ holds, we have
\begin{align}
    \min_{j \in \Lz} \left| \a_j^T \b \right|
&=  \min_{j \in \Lz} \left| x_j + \a_j^T \w + \sum_{i \in \Lz\backslash\{j\}} x_i \a_j^T \a_i \right| \notag\\
&>  \xmin - \tau - (s-1)\mu\xmax \notag\\
&=  \xmin - (2s-1)\mu\xmax - \tau + s\mu\xmax.
\end{align}
Again using \eqref{eq:th prf 1}, we obtain
\beq
\min_{j \in \Lz} \left| \a_j^T \b \right| > \xmin - (2s-1)\mu\xmax - 2\tau + \max_{j \notin \Lz} |\a_j^T\b|.
\eeq
Consequently, under the assumption \eqref{eq:xminmax cond}, we conclude that \eqref{eq:xminmax res} holds, as required.
\end{proof}

The following lemma bounds the performance of the oracle estimator under the event $B$. The usefulness of this lemma stems from the fact that, if either OMP or the thresholding algorithm correctly identify the support of $\x_0$, then their estimate is identical to that of the oracle.

\begin{lemma} \label{le:err}
Let $\x_0$ be a vector with support $\Lz = \supp(\x_0)$, and let $\b = \A\x_0 + \w$ for some noise vector $\w$. If the event $B$ of \eqref{eq:def B} occurs, then
\beq
\|\xo - \x_0\|_2^2 \le 2 s \sigma^2 (1+\alpha) \log m \frac{1}{(1 - (s-1)\mu)^2}.
\eeq
\end{lemma}

\begin{proof}
Note that both $\xo$ and $\x_0$ are supported on $\Lz$, and therefore
\begin{align}
\|\xo - \x_0\|_2^2 = \|\ALz^\pinv\b - \xzLz\|_2^2
\end{align}
where $\xzLz$ is the subvector of nonzero entries of $\x_0$. We thus have, under the event $B$,
\begin{align}
     \|\xo - \x_0\|_2^2
&=   \|\ALz^\pinv \ALz \xzLz + \ALz^\pinv \w - \xzLz\|_2^2 \notag\\
&=   \|\ALz^\pinv \w\|_2^2 \notag\\
&=   \left\| (\ALz^T \ALz)^{-1} \ALz^T \w \right\|_2^2 \notag\\
&\le \left\|(\ALz^T\ALz)^{-1}\right\|^2 \sum_{i \in \Lz} (\a_i^T \w)^2 \notag\\
&\le \frac{1}{(1 - (s-1)\mu)^2} s \sigma^2 2 (1+\alpha) \log m
\end{align}
where, in the last step, we used the definition \eqref{eq:def B} of $B$ and the fact that $\|\ALz^T\ALz\| \ge 1 - (s-1)\mu$, which was demonstrated in Appendix~\ref{ap:main}. This completes the proof the lemma.
\end{proof}

We are now ready to prove Theorem~\ref{th:omp}. The proof for the thresholding algorithm is obtained by combining the three lemmas presented above. Indeed, Lemma~\ref{le:prob B} ensures that the event $B$ occurs with probability at least as high as the required probability of success \eqref{eq:omp prob}. Whenever $B$ occurs, we have by Lemma~\ref{le:xmax} that the atoms corresponding to $\Lz$ all have strictly higher correlation with $\b$ than the off-support atoms, so that the thresholding algorithm identifies the correct support $\Lz$, and is thus equivalent to the oracle estimator $\xo$ as long as $B$ holds. Finally, by Lemma~\ref{le:err}, identification of the true support $\Lz$ guarantees the required error \eqref{eq:omp err}.

We now prove the OMP performance guarantee. Our aim is to show that when $B$ occurs, OMP correctly identifies the support of $\x_0$; the result then follows by Lemmas \ref{le:prob B} and \ref{le:err}. To this end we employ the technique used in the proof of \cite[Th.~5.1]{donoho06}. We begin by examining the first iteration of the OMP algorithm, in which one identifies the atom $\a_i$ whose correlation with $\b$ is maximal. Note that \eqref{eq:omp given} implies \eqref{eq:xmax cond}, and therefore, by Lemma~\ref{le:xmax}, the atom having the highest correlation with $\b$ corresponds to an element in the support $\Lz$ of $\x_0$. Consequently, the first step of the OMP algorithm correctly identifies an element in $\Lz$.

The proof now continues by induction. Suppose we are currently in the $i$th iteration of OMP, with $1 < i \le s$, and assume that atoms from the correct support were identified in all $i-1$ previous steps. Referring to the notation used in the definition of OMP in Section~\ref{ss:techniques}, this implies that $\supp(\xomp^{i-1}) = \Lambda^{i-1} \subset \Lz$. The $i$th step consists of identifying the atom $\a_i$ which is maximally correlated with the residual $\r^i$. By the definition of $\r^i$, we have
\beq \label{eq:r as meas}
\r^i = \A\tilde{\x}^{i-1} + \w
\eeq
where $\tilde{\x}^{i-1} = \x_0 - \xomp^{i-1}$. Thus $\supp(\tilde{\x}^{i-1}) \subseteq \Lz$, so that $\r^i$ is a noisy measurement of the vector $\A\tilde{\x}^{i-1}$, which has a sparse representation consisting of no more than $s$ atoms. Now, since
\beq
\|\xomp^{i-1}\|_0 = i-1 < s = \|x_0\|_0,
\eeq
it follows that at least one nonzero entry in $\tilde{\x}^{i-1}$ is equal to the corresponding entry in $\x_0$. Consequently
\beq \label{eq:r xmax}
\max_i |\tilde{x}^{i-1}_i| \ge \xmin.
\eeq
Note that the model \eqref{eq:r as meas} is precisely of the form \eqref{eq:b=Ax+w}, with $\r^i$ taking the place of the measurements $\b$ and $\tilde{\x}^{i-1}$ taking the place of the sparse vector $\x_0$. It follows from \eqref{eq:r xmax} and \eqref{eq:omp given} that this model satisfies the requirement \eqref{eq:xmax cond}. Consequently, by Lemma~\ref{le:xmax}, we have that under the event $B$,
\beq
\max_{i \in \Lz} |\a_i^T \r^i| > \max_{i \notin \Lz} |\a_i^T \r^i|.
\eeq
Therefore, the $i$th iteration of OMP will choose an element within $\Lz$ to add to the support. By induction it follows that the first $s$ steps of OMP all identify elements in $\Lz$, and since OMP never chooses the same element twice, the entire support $\Lz$ will be identified after $s$ iterations. This completes the proof of Theorem~\ref{th:omp}.

\bibliographystyle{IEEEtran}
\bibliography{IEEEabrv,zvika}

\end{document}